\documentclass{amsart}
\usepackage{amsmath}
\usepackage{amssymb}
\usepackage{amsthm}
\usepackage{enumerate}
\usepackage[dvips]{graphicx}
\theoremstyle{definition}
\newtheorem{definition}{Definition}[section]
\theoremstyle{plain}
\newtheorem{lemma}[definition]{Lemma}
\newtheorem{theorem}[definition]{Theorem}

\theoremstyle{remark}
\newtheorem{remark}[definition]{Remark}

\newtheorem{example}[definition]{Example}

\newcommand{\myds}{\operatorname{D}_\Sigma}
\newcommand{\myint}{\operatorname{int}}
\newcommand{\mydc}{\text{DC}}
\newcommand{\mylufs}{\text{SLUF}}
\newcommand{\mylufo}{\text{LUF}_1}
\newcommand{\myluft}{\text{LUF}_2}
\newcommand{\mypd}{\operatorname{proj.dim}}

\makeatletter
\@namedef{subjclassname@2010}{
  \textup{2010} Mathematics Subject Classification}
\makeatother

\begin{document}

\title[uniformly locally o-minimal  open core]{Uniformly locally o-minimal open core}
\author[M. Fujita]{Masato Fujita}
\address{Department of Liberal Arts,
Japan Coast Guard Academy,
5-1 Wakaba-cho, Kure, Hiroshima 737-8512, Japan}
\email{fujita.masato.p34@kyoto-u.jp}

\begin{abstract}
This paper discusses  sufficient conditions for a definably complete densely linearly ordered expansion of an abelian group  having the uniformly locally o-minimal open cores of the first/second kind and strongly locally o-minimal open core, respectively.
\end{abstract}



\maketitle

\section{Introduction}\label{sec:intro}
The open core of a structure is its reduct generated by its definable open sets. 
Dolich et al.\ first introduced the notion of open core and gave a sufficient condition for the structure having an o-minimal open core in \cite{DMS}.
Fornasiero also investigated necessary and sufficient conditions for a definably complete expansion of an ordered field having a locally o-minimal open core in \cite{F}.

A uniformly locally o-minimal structure was first introduced in \cite{KTTT} and a systematic study was made in \cite{Fuji}.
The purpose of this paper is to give sufficient conditions for structures having uniformly locally o-minimal open cores of the first/second kind and having strongly locally o-minimal open core, respectively.
The following theorem is our main theorem.
The notations in the theorem are defined in Section \ref{sec:def}.
\begin{theorem}\label{thm:main}
Consider an expansion of a densely linearly ordered abelian group $\mathcal R=(R,<,+,0,\ldots)$ with $\mathcal R \models \mydc$.
\begin{itemize}
\item If $\mathcal R \models \mylufs$, the structure $\mathcal R$ has a strongly locally o-minimal open core. 
\item If $\mathcal R \models \mylufo$, the structure $\mathcal R$ has a uniformly locally o-minimal open core of the first kind. 
\item If $\mathcal R \models \myluft$, the structure $\mathcal R$ has a uniformly locally o-minimal open core of the second kind. 
\end{itemize}
\end{theorem}
We can prove the theorem basically following the same strategy as the proof of \cite[Theorem A]{DMS}.
We first review the notion of $\myds$-families and $\myds$-sets used in the proof of \cite[Theorem A]{DMS} in Section \ref{sec:preliminaries}.
The key lemma is that, for a $\myds$-family $\{X_{r,s}\}_{r>0,s>0}$, the set $X_{r,s}$ has a nonempty interior for some $r>0$ and $s>0$ when the $\myds$-set $X=\bigcup_{r>0,s>0}X_{r,s}$ has a non empty interior.
This lemma holds true both for the o-minimal open core case \cite[3.1]{DMS} and our case.
However, a new investigation is necessary to demonstrate the lemma in our case.
Basic lemmas including the above lemma are proved in Section \ref{sec:basic}.
We cannot  use the same definition of dimension of $\myds$-sets as \cite{DMS}.
We give a new definition of dimension in Section \ref{sec:preliminaries}.
We finally prove the theorem in Section \ref{sec:open_core} using this concept.

We introduce the terms and notations used in this paper.
The term `definable' means `definable in the given structure with parameters' in this paper.
A \textit{CBD} set is a closed, bounded and definable set.
A \textit{CDD} set is a closed, discrete and definable set.
For any set $X \subset R^{m+n}$ definable in a structure $\mathcal R=(R,\ldots)$ and for any $x \in R^m$, the notation $X_x$ denotes the fiber defined as $\{y \in R^n\;|\; (x,y) \in X\}$.
For a linearly ordered structure $\mathcal R=(R,<,\ldots)$, an open interval is a definable set of the form $\{x \in R\;|\; a < x < b\}$ for some $a,b \in R$.
It is denoted by $(a,b)$ in this paper.
We define a closed interval in the same manner and it is denoted by $[a,b]$.
An open box in $R^n$ is the direct product of $n$ open intervals.
A closed box is defined similarly.
Let $A$ be a subset of a topological space.
The notations $\myint(A)$ and $\overline{A}$ denote the interior and the closure of the set $A$, respectively.
The notation $|S|$ denotes the cardinality of a set $S$.

\section{Definitions}\label{sec:def}
We review the definitions and the assertions introduced in the previous studies.
A \textit{constructible} set is a finite boolean combination of open sets.
Note that every constructible definable set is a finite boolean combination of open definable sets by \cite{DM}.
The definition of a definably complete structure is found in \cite{M}.
The notation $\mathcal R \models \mydc$ means that the structure $\mathcal R$ is definably complete.
A locally o-minimal structure is defined and investigated in \cite{TV}.
Readers can find the definitions of uniformly locally o-minimal structures of the first/second kind in \cite{Fuji}.
The open core of a structure is defined in \cite{DMS}.

We give the definitions used in Theorem \ref{thm:main}.

\begin{definition}[Local uniform finiteness]
A densely linearly ordered structure $\mathcal R=(R,<,\ldots)$ satisfies \textit{locally uniform finiteness of the second kind} -- for short, $\mathcal R \models \myluft$ if, for any definable subset $X \subset R^{m+1}$, $a \in R$ and $b \in R^m$, there exist 
\begin{itemize}
\item a positive integer $N$, 
\item a closed interval $I$ with $a \in \myint(I)$ and 
\item an open box $U \subset R^m$ with $b \in U$ 
\end{itemize}
such that $|X_x \cap J|=\infty$ or $|X_x \cap J| \leq N$ for all $x \in U$ and all closed intervals $J$ with $a \in \myint(J) \subset I$.

An easy induction shows that $\mathcal R \models \myluft$ if and only if, for any definable subset $X \subset R^{m+n}$ and $(b,a) \in R^{m+n}$, there exist a positive integer $N$, a closed box $B$ with $a \in \myint(B)$ and an open box $U$ with $b \in U$ such that $|X_x \cap B'|=\infty$ or $|X_x \cap B'| \leq N$ for all $x \in U$ and all closed boxes $B'$ with $a \in \myint(B') \subset B$.

The structure $\mathcal R=(R,<,\ldots)$ satisfies \textit{locally uniform finiteness of the first kind} -- for short, $\mathcal R \models \mylufo$ if $\mathcal R \models \myluft$ and we can take $U=R^m$ in the definition of $\myluft$.
The structure $\mathcal R=(R,<,\ldots)$ satisfies \textit{strongly locally uniform finiteness} -- for short, $\mathcal R \models \mylufs$ if $\mathcal R \models \mylufo$ and we can take $I$ independently of the definable set $X \subset R^{m+1}$.
\end{definition}

\begin{remark}
A definably complete uniformly locally o-minimal structure of the second kind satisfies  locally uniform finiteness of the second kind by \cite[Theorem 4.2]{Fuji}.
\end{remark}

\begin{remark}
Consider a locally o-minimal expansion of the group of reals $\widetilde{\mathbb R}=(\mathbb R, <,0,+,\ldots)$.
The following assertion is \cite[Theorem 4.3]{Fuji2}.
\begin{quotation}
For any definable subset $X$ of $\mathbb R^{n+1}$, there exist a positive element $r \in \mathbb R$ and a positive integer $K$ such that, for any $a \in \mathbb R^n$, the definable set $X \cap (\{a\} \times (-r,r))$ has at most $K$ connected components.
\end{quotation}
By reviewing its proof, it is easy to check that the above $r$ can be taken independently of $X$ because $\widetilde{\mathbb R}$ is strongly locally o-minimal by \cite[Corollary 3.4]{TV}.
Therefore, we have $\widetilde{\mathbb R} \models \mylufs$.
\end{remark}

\begin{example}
We give an example of a structure which has a strongly locally o-minimal open core, but is not a locally o-minimal.

We first consider the language $L_1=\{0,1,+,-,<\}$.
We define $L_1$-structures $\widetilde{\mathbb Q}=(\mathbb Q, 0^{\widetilde{\mathbb Q}},1^{\widetilde{\mathbb Q}},+^{\widetilde{\mathbb Q}},-^{\widetilde{\mathbb Q}},<^{\widetilde{\mathbb Q}})$ and $\widetilde{\mathbb R}=(\mathbb R, 0^{\widetilde{\mathbb R}},1^{\widetilde{\mathbb R}},+^{\widetilde{\mathbb R}},-^{\widetilde{\mathbb R}},<^{\widetilde{\mathbb R}})$ naturally.
It is easy to demonstrate that both $\widetilde{\mathbb Q}$ and $\widetilde{\mathbb R}$ have quantifier elimination.
We can also easy to demonstrate that $\widetilde{\mathbb Q}$ is an elementary substructure of $\widetilde{\mathbb R}$ using the Tarski-Vaught test.
They are both o-minimal structures.

The structures $\widetilde{[0,1)_{\mathbb Q}}$ and $\widetilde{[0,1)_{\mathbb R}}$ are the restrictions of $\widetilde{\mathbb Q}$ and $\widetilde{\mathbb R}$ to the sets $[0,1)_{\mathbb Q}=\{x \in \mathbb Q\;|\; 0 \leq x < 1\}$ and $[0,1)_{\mathbb R}=\{x \in \mathbb R\;|\; 0 \leq x < 1\}$ defined in \cite[Definition 2]{KTTT}, respectively.
The structure $\widetilde{[0,1)_{\mathbb Q}}$ is again an elementary substructure of $\widetilde{[0,1)_{\mathbb R}}$.
The notation $\mathcal M = \left(\widetilde{[0,1)_{\mathbb R}}, \widetilde{[0,1)_{\mathbb Q}}\right)$ denotes their dense pair. 
The definition of dense pairs is found in \cite{vdD2}.
The dense pair $\mathcal M$ satisfies uniform finiteness by \cite[Corollary 4.5]{vdD2}.
The notion of uniform finiteness is introduced in \cite{DMS}.

We next consider the language $L=\{0,1,+,-,<, P_{\mathbb Z}, P_{\mathbb Q}\}$, where $P_{\mathbb Z}$ and $P_{\mathbb Q}$ are unary predicates.
We define an $L$-structure $\mathcal R=(\mathbb R, 0^{\mathcal R},1^{\mathcal R},+^{\mathcal R},-^{\mathcal R},<^{\mathcal R}, P_{\mathbb Z}^{\mathcal R}, P_{\mathbb Q}^{\mathcal R})$ as follows:
\begin{itemize}
\item $\mathcal R \models P_{\mathbb Z}^{\mathcal R}(x)$ if and only if $x \in \mathbb Z$;
\item $\mathcal R \models P_{\mathbb Q}^{\mathcal R}(x)$ if and only if $x \in \mathbb Q$.
\end{itemize}
The following claim is proved by the induction on the complexity of the formula defining the definable set $X$. We omit the proof.

\medskip
\textbf{Claim.} Let $X$ be a subset of $\mathbb R^n$ definable in $\mathcal R$.
There exist finite subsets $X_1, \ldots X_k$ of $[0,1)_{\mathbb R}$ definable in the dense pair $\mathcal M$ and a map $\iota:\mathbb Z^n \rightarrow \{1, \ldots, k\}$ such that, for any $z =(z_1, \ldots, z_n) \in \mathbb Z^n$, we have 
\begin{center}
$X \cap \left(\displaystyle\prod_{i=1}^n [z_i,z_i+1)\right)= z + X_{\iota(z)}$, 
\end{center}
where $[c,d)=\{x \in R\;|\; c \leq x < d\}$ and $z + X_{\iota(z)} = \{(x_1, \ldots, x_n) \in \mathbb R^n\;|\; (x_1-z_1,\ldots, x_n-z_n) \in X_{\iota(z)}\}$.
\medskip

The structure $\mathcal R$ is not locally o-minimal because the set $\mathbb Q$ is definable in $\mathcal R$.
We have $\mathcal R \models \mylufs$ by the above claim because $\mathcal M$ satisfies uniform finiteness.
We also have $\mathcal R \models \mydc$ because the universe $\mathbb R$ is complete.
The structure $\mathcal R$ has a strongly locally o-minimal open core by Theorem \ref{thm:main}.
\end{example}

\section{$\myds$-sets and its dimension}\label{sec:preliminaries}
We can apply the same strategy as \cite{DMS} to our problem.
Dolich et al.\ used the notion of $\myds$-sets in \cite{DMS}.
They play an important role also in this paper.
\begin{definition}[$\myds$-sets]
Consider an expansion of a densely linearly ordered abelian group $\mathcal R=(R,<,+,0\ldots)$.
A \textit{parameterized family} of definable sets are the family of the fibers of a definable set.
A parameterized family $\{X_{r,s}\}_{r>0,s>0}$ of CBD subsets of $R^n$ is called a \textit{$\myds$-family} if $X_{r,s} \subset X_{r',s}$ and  $X_{r,s'} \subset X_{r,s}$ whenever $r \leq r'$ and $s \leq s'$.
A definable subset $X$ of $R^n$ is a \textit{$\myds$-set} if $X = \displaystyle\bigcup_{r>0,s>0} X_{r,s}$ for some $\myds$-family $\{X_{r,s}\}_{r>0,s>0}$.
\end{definition}

The following two lemmas are found in \cite{DMS,M}.
\begin{lemma}\label{lem:quo}
Consider an expansion of a densely linearly ordered abelian group $\mathcal R$ with $\mathcal R \models \mydc$.
The following assertions are true:
\begin{enumerate}[(1)]
\item The projection image of a $\myds$-set is $\myds$.
\item Fibers, finite unions and finite intersections of $\myds$-sets are $\myds$.
\item Every constructible definable set is $\myds$.
\end{enumerate}
\end{lemma}
\begin{proof}
(1) Immediate from \cite[Lemma 1.7]{M}.\ 
(2) \cite[1.9(1)]{DMS}.\ 
(3) \cite[1.10(1)]{DMS}.
\end{proof}

\begin{lemma}\label{lem:interior0}
Let $\mathcal R=(R,<,+,0,\ldots)$ be an expansion of a densely linearly ordered abelian group  with $\mathcal R \models \mydc$.
A CBD set $X \subset R^{n+1}$ has a nonempty interior if the CBD set 
\begin{center}
$\{x \in R^n\;|\; X_x \text{ contains a closed interval of length }s\}$
\end{center}
has a nonempty interior for some $s>0$.
\end{lemma}
\begin{proof}
\cite[2.8(2)]{DMS}
\end{proof}

The notion of dimension used for o-minimal open cores in \cite{DMS} is not appropriate for our setting.
We give a new definition of dimension of a $\myds$-set.
The dimension of a set definable in a locally o-minimal structure admitting local definable cell decomposition is defined in \cite[Section 5]{Fuji}.
In a definably complete uniformly locally o-minimal structure of the second kind, the dimension defined below coincides with the dimension defined in \cite[Section 5]{Fuji} by \cite[Corollary 5.3]{Fuji}. 
\begin{definition}[Dimension]
Let $\mathcal R=(R,<,\ldots)$ be an expansion of a densely linearly ordered structure.
Consider a $\myds$-subset $X$ of $R^n$ and a point $x \in R^n$.
The \textit{local dimension $\dim_xX$ of $X$ at $x$} is defined as follows:
\begin{itemize}
\item $\dim_xX=-\infty$ if there exists an open box $B$ with $x \in B$ and $B \cap X=\emptyset$.
\item Otherwise, $\dim_xX$ is the supremum of nonnegative integers $d$ such that, for any open box $B$ with $x \in B$, the image $\pi(B \cap X)$ has a nonempty interior for some coordinate projection $\pi:R^n \rightarrow R^d$.
\end{itemize} 
The \textit{dimension of $X$} is defined by $\dim X = \sup\{\dim_xX\;|\; x \in R^n\}$.
The \textit{projective dimension $\mypd X$ of $X$} is defined as follows:
\begin{itemize}
\item $\mypd X=-\infty$ if $X$ is an empty set.
\item Otherwise, $\mypd X$ is the supremum of nonnegative integers $d$ such that the image $\pi(X)$ has a nonempty interior for some coordinate projection $\pi:R^n \rightarrow R^d$.
\end{itemize} 
\end{definition}

The following lemma illustrates that the dimension and the projective dimension coincides in some open box.
\begin{lemma}\label{lem:dim_lem1}
Let $\mathcal R=(R,<,\ldots)$ be an expansion of a densely linearly ordered structure.
Consider a $\myds$-subset $X$ of $R^n$ of dimension $d$.
Take a point $x \in R^n$ with $\dim_xX=d$.
We have $\dim(X \cap B) = \mypd(X \cap B)=d$ for any sufficiently small open box $B$ in $R^n$ with $x \in B$.
\end{lemma}
\begin{proof}
Since $\dim_xX=d$, the projection image $\pi(B \cap X)$ has an empty interior for any coordinate projection $\pi:R^n \rightarrow R^{d+1}$ when $B$ is a sufficiently small open box with $x \in B$.
In particular, $\mypd B \cap X \leq d$.
It is obvious that $d = \dim_xX \leq \dim B \cap X \leq \dim X =d$.
We have shown that $\dim B \cap X = d$ and $\mypd B \cap X \leq \dim B \cap X$.
The opposite inequality $\dim B \cap X \leq \mypd B \cap X$ is obvious from the definition.
\end{proof}

\section{Basic lemmas}\label{sec:basic}
We introduce basic lemmas in this section.
We first prove two lemmas.
We can prove the lemmas by localizing the arguments in \cite[2.4]{DMS}.

\begin{lemma}\label{lem:base1}
Let $\mathcal R=(R,<,+,0,\ldots)$ be an expansion of a densely linearly ordered abelian group  with $\mathcal R \models \mydc, \myluft$.

For any definable set $X \subset R^{m+n}$ and a point $(b,a) \in R^m \times R^n$, there exist a positive integer $N$, a closed box $B$ with $a \in \myint(B)$ and an open box $U$ containing the point $b$ such that, if $X_x \cap B$ is discrete, we have $|X_x \cap B | \leq N$ for any $x \in U$.

In addition, we can take $U=R^m$ if $\mathcal R \models \mylufo$.
We can take $B$ independently of $X \subset R^{m+n}$ if $\mathcal R \models \mylufs$.
\end{lemma}
\begin{proof}
We first demonstrate the following claim:
\medskip

\textbf{Claim.} Let $X$ be a definable subset of $R^{m+n}$.
Assume that the fiber $X_x$ is CDD for any $x \in R^m$.
For any $(b,a) \in R^m \times R^n$, there exists a closed box $B$ with $a \in \myint(B)$ and an open box $U$ containing the point $b$ such that $X_x \cap B$ is a finite set for any  $x \in U$.
In addition, we can take $U=R^m$ if $\mathcal R \models \mylufo$.
We can take $B$ independently of $X \subset R^{m+n}$ if $\mathcal R \models \mylufs$.
\medskip

We prove the claim. 
Assume the contrary.
We can find a point $(b,a) \in R^m \times R^n$ such that, for any closed box $B$ with $a \in \myint(B)$ and any open box $U$ with $b \in U$, $X_x \cap B$ is infinite for some $x \in U$.
Consider the set
\begin{align*}
Y = \{(x,y_1,y_2) &\in R^m \times R^n \times R^n\;|\; (x,y_1) \in X,\ (x,y_2) \in X,\  \\
&y_1 \geq y_2 \text{ in the lexicographic order}\}\text{.}
\end{align*}
Since $\mathcal R \models \myluft$, there exist a positive integer $N$, a closed box $B$ with $a \in \myint(B)$ and an open box $V$ with $(b,a) \in V=U_1 \times U_2 \subset R^m \times R^n$ such that, for any $(x,y_1) \in V$, we have $|Y_{(x,y_1)} \cap B|=\infty$ or $|Y_{(x,y_1)} \cap B| \leq N$.
Shrinking $B$ if necessary, we may assume that $B$ is contained in $U_2$.
Fix such a closed box $B$.
Take a point $x \in U_1$ such that $X_x \cap B$ is infinite.
Such a point $x$ exists by the assumption.
Note that $X_x \cap B$ is CBD.
We construct a point $z_k \in B$ with $|Y_{(x,z_k)} \cap B| = k$ inductively.
Take $z_1 = \operatorname{\mathbf{lexmin}}(X_x \cap B)$, then $Y_{(x,z_1)} \cap B=\{z_1\}$.
The notation $\operatorname{\mathbf{lexmin}}$ denotes the lexicographic minimum defined in \cite{M}.
Take $z_2 = \operatorname{\mathbf{lexmin}}(X_x \cap B \setminus \{z_1\})$, then $Y_{(x,z_2)} \cap B=\{z_1,z_2\}$.
Take $z_3, z_4, \ldots$ in this manner.
We have $|Y_{(x,z_{N+1})} \cap B|=N+1$, which is a contradiction.

It is obvious that $U=R^m$ if $\mathcal R \models \mylufo$ and $B$ is common to all $X$ if  $\mathcal R \models \mylufs$.
We have finished the proof of the claim.
\medskip

We return to the proof of the lemma.
Take a bounded closed box $C$ with $b \in \myint(C)$. 
Set $D=\{(s,x,y) \in  R \times C \times R^n \;|\; \prod_{i=1}^n[y_i-s,y_i+s]^n \cap X_x = \{y\}\}$, where $y=(y_1, \ldots, y_n)$.
The following assertions are trivial.
\begin{itemize}
\item The fiber $D_{(s,x)}$ is CDD.
\item $\bigcup_{s>0}D_{s}=\{(x,y)\;|\;y \text{ is a discrete point in }X_x\}$.
\end{itemize}
Apply the claim to the set $D$.
We can take an open box $U$ with $b \in U$ and a closed box $B$ with $a \in \myint(B)$ such that $D_{(s,x)} \cap B$ are finite sets for all $x \in U$ and all sufficiently small $s>0$.
Since $\mathcal R \models \myluft$, shrinking $B$ and $U$ if necessary, we have $\left|D_{(s,x)} \cap B \right| \leq N$ for some positive integer $N$, any $x \in U$ and any sufficiently small $s>0$.
Since $D_{(s,x)} \subset D_{(s',x)}$ for all $s>s'>0$, we have $\left|\bigcup_{s>0}D_{(s,x)}  \cap B\right| \leq N$.
The `in addition' parts of the lemma are obvious.
\end{proof}

\begin{lemma}\label{lem:base2}
Let $\mathcal R=(R,<,+,0,\ldots)$ be an expansion of a densely linearly ordered abelian group  with $\mathcal R \models \mydc, \myluft$.

For any definable set $X \subset R^{m+1}$ and a point $(b,a) \in R^m \times R$, there exist positive integers $N_1$, $N_2$, an open interval $I$ with $a \in I$ and an open box $U$ containing the point $b$ such that, for any $x \in U$, 
\begin{itemize}
\item the open set $\myint(X_x) \cap I$ is the union of at most $N_1$ open intervals in $I$;
\item the closed set $\overline{X_x} \cap \overline{I}$ is the union of at most $N_1$ points and $N_2$ closed intervals in $\overline{I}$.
\end{itemize}

In addition, we can take $U=R^m$ if $\mathcal R \models \mylufo$.
We can take $I$ independently of $X \subset R^{m+1}$ if $\mathcal R \models \mylufs$.
\end{lemma}
\begin{proof}
The assertion on the closure follows the assertion on the interior by considering $R^{m+1} \setminus X$ in place of $X$.
We only prove the latter.
We may assume that $X_x$ is open for any $x \in R^m$ without loss of generality.
Consider the set
\begin{align*}
C=\{(r,x,y) \in R \times R^m \times R\;|\; &r>0, \exists \varepsilon>0,\ (y-\varepsilon, y+\varepsilon) \subset X_x \cap (a-r, a+r)\\
& y-\varepsilon, y+\varepsilon \in \partial (X_x \cap (a-r, a+r))\}\text{,}
\end{align*}
where $\partial (X_x \cap (a-r, a+r))$ denotes the boundary of the set $X_x \cap (a-r, a+r)$ in $R$.
The fiber $C_{(r,x)}$ is discrete.
By Lemma \ref{lem:base1}, there exist a positive integer $N$, a positive element $s>0$, a closed interval $I'$ with $a \in \myint(I')$ and an open box $U$ containing the point $b$ such that $|C_{(r,x)} \cap I'| \leq N$ for all $x \in U$ and $0<r<s$.
Take a sufficiently small $r>0$ with $[a-r,a+r] \subset I'$ and set $I=(a-r,a+r)$.
The definable set $X_x \cap I$ consists of at most $N$ open intervals for any $x \in U$.
\end{proof}

Lemma \ref{lem:bb1} is the key lemma introduced in Section \ref{sec:intro}.
We prove the lemma combining the arguments of \cite{DMS} and \cite{Fuji3}.
Lemma \ref{lem:bb2} is another key lemma corresponding to \cite[Theorem 3.3]{Fuji}.
They are proved simultaneously.
\begin{lemma}\label{lem:bb1}
Let $\mathcal R=(R,<,+,0,\ldots)$ be an expansion of a densely linearly ordered abelian group  with $\mathcal R \models \mydc, \myluft$.
Let $\{X_{r,s} \subset R^n\}_{r>0,s>0}$ be a $\myds$-family. 
Set $X=\bigcup_{r,s}X_{r,s}$.
One of the following conditions is satisfied:
\begin{itemize}
\item The $\myds$-set $X$ has an empty interior and it is locally finite when $n=1$. 
\item The CBD set $X_{r,s}$ has a nonempty interior for some $r>0$ and $s>0$.
\end{itemize}
\end{lemma}

\begin{lemma}\label{lem:bb2}
Let $\mathcal R=(R,<,+,0,\ldots)$ be an expansion of a densely linearly ordered abelian group  with $\mathcal R \models \mydc, \myluft$.
Consider a $\myds$-subset $X$ of $R^n$ with a nonempty interior.
Let $X=X_1 \cup X_2$ be a partition into two $\myds$-sets.
At least one of $X_1$ and $X_2$ has a nonempty interior.  
\end{lemma}

\begin{proof}
We prove the lemmas by the induction on $n$.
We first consider the case in which $n=1$.
We first demonstrate Lemma \ref{lem:bb1}.
Assume that $\myint(X_{r,s})=\emptyset$ for all $r>0$ and $s>0$.
We have only to show that $X$ is locally finite. 
Fix an arbitrary point $a \in R$ and $r>0$.
Set $X_r=\bigcup_{s>0}X_{r,s}$.
There exist a closed interval $I$ with $a \in \myint(I)$, a positive element $t \in R$, positive integers $N_1$ and $N_2$ such that the intersection $I \cap X_{r,s}$ is the union of at most $N_1$ points and $N_2$ closed intervals for any $0<s<t$ by Lemma \ref{lem:base2}. 
The set $I \cap X_{r,s}$ consists of at most $N_1$ points because $\myint(X_{r,s})=\emptyset$.
We have $|X_r \cap I| \leq N_1$ because $\{X_{r,s}\}_{s>0}$ is a decreasing sequence.
The set $X_r$ is a CDD set.

We show that $X=\bigcup_{r>0}X_r$ is a CDD set.
Assume the contrary.
There exists a point $a \in R$ such that, for any open interval $I$ containing a point $a$, the definable set $I \cap X$ is an infinite set.
We may assume that $\{x > a \} \cap I \cap X$ is an infinite set for any open interval $I$ containing the point $a$  without loss of generality.

Consider the definable function $f:\{r \in R\;|\;r >0\} \rightarrow \{x \in R\;|\; x  > a\}$ defined by $f(r)= \inf \{x >a \;|\; x \in X_r\}$.
As in the proof of \cite[Theorem 4.3]{Fuji3}, we can prove the following assertions:
\begin{itemize}
\item The definable function $f$ is a decreasing function and $\lim_{r \to \infty}f(r)=a$. 
\item Consider the image $\operatorname{Im}(f)$ of the function $f$.
For any $b \in \operatorname{Im}(f)$, there exists a point $b_1 \in \operatorname{Im}(f)$ such that $b<b_1$ and the open interval $(b,b_1)$ has an empty intersection with $\operatorname{Im}(f)$.
\end{itemize}
We may assume that the intersection $\overline{\operatorname{Im}(f)} \cap \overline{I}$ of the closure of the image with $\overline{I}$ consists of finite points and finite closed intervals by Lemma \ref{lem:base2} shrinking the open interval $I$ if necessary.
We lead to a contradiction assuming that it contains a closed interval $J$.
Take an arbitrary point $b \in \operatorname{Im}(f)$ in the interior of the closed interval $J$.
The open interval $(b,b_1)$ has an empty intersection with $\operatorname{Im}(f)$ for some $b_1 \in R$.
It is a contradiction.
We have shown that $I \cap \operatorname{Im}(f)$ is a finite set.
It is a contradiction to the fact that $\lim_{r \to \infty}f(r)=a$. 
We have shown that $X=\bigcup_{r>0}X_r$ is a CDD set.
It is obviously locally finite.
We have demonstrated Lemma \ref{lem:bb1} when $n=1$.

Lemma \ref{lem:bb2} is immediate from Lemma \ref{lem:bb1} when $n=1$.

We next consider the case in which $n>1$.
We first demonstrate Lemma \ref{lem:bb1}.
We can prove that $X_r=\bigcup_{s>0}X_{r,s}$ has an empty interior if $X_{r,s}$ have empty interiors for all $s>0$ in the same way as \cite[Lemma 3.3]{Fuji3}.
Now we can get Lemma \ref{lem:bb1} in the same way as the proof of \cite[Lemma 4.1, Lemma 4.2]{Fuji3} using Lemma \ref{lem:bb2} for $n-1$ instead of \cite[Theorem 3.3]{Fuji}. 

The remaining task is to prove Lemma \ref{lem:bb2} when $n>1$.
Take $\myds$-families $\{X_{r,s}^i\}_{r,s}$ with $X_i = \bigcup_{r,s}X_{r,s}^i$ for $i=1,2$.
Set $X_{r,s}=X_{r,s}^1 \cup X_{r,s}^2$.
It is a CBD set.
We have $\myint(X_{r,s}) \not= \emptyset$ for some $r>0$ and $s>0$ by Lemma \ref{lem:bb1} for $n$ because $X=\bigcup_{r,s} X_{r,s}$.
If at least one of $X_{r,s}^1$ and $X_{r,s}^2$ has a nonempty interior, at least one of $X_1$ and $X_2$ has a nonempty interior.
Therefore, we may assume that $X_1$ and $X_2$ are CBD sets.
Let $B$ be a closed box contained in $X$.
We have $B=(X_1 \cap B) \cup (X_2 \cap B)$.
If the lemma is true for $B$, the lemma is also true for the original $X$.
Hence, we may assume that $X$ is a closed box.

Let $\pi$ be the coordinate projection forgetting the last coordinate.
For $i=1,2$ and $s>0$, we set
\begin{align*}
S_s^i=\{x \in \pi(X_i)\;|\;\text{ the fiber } (X_i)_x \text{ contains a closed interval of length }s\}
\end{align*}
They are CBD sets.
Set $T_i = \bigcup_{s>0} S_s^i$, which is $\myds$.
It is obvious that $T_i=\{x \in \pi(X_i)\;|\;  (X_i)_x \text{ contains an open interval}\}$.
Since $X$ is a closed box, we have $\pi(X) = T_1 \cup T_2$ by Lemma \ref{lem:bb2} for $n=1$.
At least one of $T_1$ and $T_2$ has a nonempty interior by the induction hypothesis.
We may assume that $\myint(T_1) \not=\emptyset$ without loss of generality.
We have $\myint(S_s^1)$ for some $s>0$ by Lemma \ref{lem:bb1} for $n-1$.
The CBD set $X_1$ has a non-empty interior by Lemma \ref{lem:interior0}.
We have finished the proof of Lemma \ref{lem:bb2}.
\end{proof}

We finally demonstrate that a definable map whose graph is $\myds$ is continuous on a dense set.
\begin{lemma}\label{lem:cont_function}
Let $\mathcal R=(R,<,+,0,\ldots)$ be an expansion of a densely linearly ordered abelian group  with $\mathcal R \models \mydc,\myluft$.
Consider a definable map $f:U \rightarrow R^n$ defined on an open set $U$ whose graph is a $\myds$-set.
There exists a nonempty open box $V$ contained in $U$ such that the restriction of $f$ to $V$ is continuous. 
\end{lemma}
\begin{proof}
We can prove the lemma in the same way as \cite[Lemma 5.1]{Fuji3}.
We use Lemma \ref{lem:bb1} instead of \cite[Lemma 3.4]{Fuji3}.
We omit the proof.
\end{proof}

\section{Uniformly locally o-minimal open core}\label{sec:open_core}
We demonstrate Theorem \ref{thm:main} in this section.
The following lemma claims that all $\myds$-subsets of $R$ are constructible.
\begin{lemma}\label{lem:dim1}
Let $\mathcal R=(R,<,+,0,\ldots)$ be an expansion of a densely linearly ordered abelian group  with $\mathcal R \models \mydc, \myluft$.
Consider a definable subset $X$ of $R$.
The followings are equivalent:
\begin{enumerate}[(1)]
\item The set $X$ is a $\myds$-set.
\item For any $x \in R$, there exists an open interval $I$ such that $x \in I$ and $X \cap I$ is a finite union of points and open intervals.
\item The set $X$ is constructible. 
\end{enumerate}
\end{lemma}
\begin{proof}
(1) $\Rightarrow$ (2):
The difference $X \setminus \myint(X)$ is $\myds$ by Lemma \ref{lem:quo}.
It is locally finite by Lemma \ref{lem:bb1}.
For any $x \in R$, the intersection $\myint(X) \cap I$ is a finite union of open intervals for some open interval $I$ with $x \in I$ by Lemma \ref{lem:base2}.
We get the assertion (2).

(2) $\Rightarrow$ (3):
The difference $X \setminus \myint(X)$ is locally finite by the assertion (2).
It means that $X$ is constructible.

(3) $\Rightarrow$ (1):
Immediate from Lemma \ref{lem:quo} (3).
\end{proof}

We next demonstrate that all $\myds$-sets satisfy a condition satisfied by the sets definable in a uniformly locally o-minimal structure of the second kind when $\mathcal R \models \mydc, \myluft$.
\begin{lemma}\label{lem:uniform2_2}
Let $\mathcal R=(R,<,+,0,\ldots)$ be an expansion of a densely linearly ordered abelian group  with $\mathcal R \models \mydc, \myluft$.
Consider a $\myds$-subset $X$ of $R^{m+1}$.
For any $a \in R$ and $b \in R^m$, there exist positive integers $N_1$, $N_2$, an open interval $I$ with $a \in I$ and an open box $B$ with $b \in B$ such that $X_x \cap I$ is the union of at most $N_1$ points and $N_2$ open intervals for any $x \in B$.

Furthermore, we can take $B=R^m$ if $\mathcal R \models \mylufo$.
We can take $I$ independently of $X \subset R^{m+1}$ if $\mathcal R \models \mylufs$.
\end{lemma}
\begin{proof}
The difference $X_x \setminus \myint(X_x)$ is discrete and closed by Lemma \ref{lem:dim1} for all $x \in R^m$.
There exist a positive integer $N_1$, an open box $B$ and a closed interval $I_1$ with $a \in \myint(I_1)$ and $|(X_x \setminus \myint(X_x)) \cap I_1| \leq N_1$ for all $x \in B$ by Lemma \ref{lem:base2}.
There exist a positive integer $N_2$ and an open interval $I_2$ with $a \in I_2$ such that $\myint(X_x) \cap I_2$ consist of at most $N_2$ open intervals for all $x \in B$ also by Lemma \ref{lem:base2}.
Set $I=\myint(I_1) \cap I_2$, then $I$ satisfies the conditions in the lemma.
The `furthermore' part is obvious by Lemma \ref{lem:base2}.
\end{proof}

We investigate $\myds$-sets of dimension zero.
\begin{lemma}\label{lem:dim0}
Let $\mathcal R=(R,<,+,0,\ldots)$ be an expansion of a densely linearly ordered abelian group  with $\mathcal R \models \mydc,  \myluft$.
A $\myds$-set of dimension zero is discrete and closed.
In particular, it is constructible.
\end{lemma}
\begin{proof}
Consider a $\myds$ subset $X$ of $R^n$ of dimension zero.
Let $\pi_i$ be the projections onto the $i$-th coordinate for all $1 \leq i \leq n$.
Let $x \in R^n$ be an arbitrary point.
Take a sufficiently small open box $B$ with $x \in B$.
The projection images $\pi_i(X \cap B)$ have empty interiors for all $i$ because $\dim_xX \leq 0$.
By Lemma \ref{lem:dim1}, we may assume that $\pi_i(X \cap B)$ is empty or a singleton by shrinking $B$ if necessary.
It means that $X$ is discrete and closed.
\end{proof}

The following three lemmas are essential parts of the proof of our main theorem. 
\begin{lemma}\label{lem:interior3}
Let $\mathcal R=(R,<,+,0,\ldots)$ be an expansion of a densely linearly ordered abelian group  with $\mathcal R \models \mydc,  \myluft$.
Consider a $\myds$-subset $X$ of $R^{n+1}$.
Set 
\begin{equation*}
\mathcal I(X)=\{x \in R^n\;|\; \text{the fiber } X_x \text{ contains an open interval}\}\text{.}
\end{equation*}
It is a $\myds$-set and we have $\mypd(\mathcal I(X)) < \mypd(X)$.
\end{lemma}
\begin{proof}
Let $\{X_{r,s}\}_{r>0,s>0}$ be a $\myds$-family with $X=\bigcup_{r>0,s>0}X_{r,s}$.
Set $Y_{r,s}=\{x \in R^n\;|\; \exists t \in R,\ [t-s,t+s] \subset (X_{r,s})_x\}$.
The set $Y_{r,s}$ is CBD.
We show that $\mathcal I(X)=\bigcup_{r,s}Y_{r,s}$.
It is obvious that $\bigcup_{r,s}Y_{r,s} \subset \mathcal I(X)$.
We demonstrate the opposite inclusion.
Take an arbitrary point $x \in \mathcal I(X)$.
We have $\myint(X_x) \not= \emptyset$.
We get $\myint((X_{r,s})_x) \not=\emptyset$ for some $r>0$ and $s>0$ by Lemma \ref{lem:bb1}.
There exist $t \in R$ and $s>0$ with $[t-s,t+s] \subset (X_{r,s})_x$.
It means that $x \in Y_{r,s}$.
We have demonstrated that $\mathcal I(X)=\bigcup_{r,s}Y_{r,s}$.
In particular, $\mathcal I(X)$ is a $\myds$-set.

We next demonstrate that $\mypd(\mathcal I(X)) < \mypd(X)$.
It is obvious when $\myint(X) \not= \emptyset$ because $\mypd(\mathcal I(X)) \leq n$ and $\mypd(X)=n+1$ by the definition.
We consider the case in which $\myint(X) = \emptyset$.
Let $\pi:R^{n+1} \rightarrow R^n$ be the coordinate projection forgetting the last coordinate.
We have $\mypd \mathcal I(X) \leq \mypd \pi(X) \leq \mypd X$ because $\mathcal I(X) \subset \pi(X)$.
We lead to a contradiction assuming that $\mypd \mathcal I(X) = \mypd X$.
Set $d=\mypd X=\mypd \mathcal I(X)$.
Take a coordinate projection $\pi_1:R^n \rightarrow R^d$ with $\myint(\pi_1(\mathcal I(X))) \not= \emptyset$.
The coordinate projection $\pi_2:R^{n+1} \rightarrow R^d$ is the composition of $\pi_1$ with $\pi$.
We have $\myint(\pi_2(X)) \not= \emptyset$ because $\pi_1(\mathcal I(X)) \subset \pi_2(X)$.
The notation $\pi_3:R^{n+1} \rightarrow R$ denotes the coordinate projection onto the last coordinate.
The coordinate projection $\Pi=(\pi_2,\pi_3):R^{n+1} \rightarrow R^{d+1}$ is given by $\Pi(x)=(\pi_2(x),\pi_3(x))$.
Consider the set $T=\{x \in R^d\;|\; \Pi(X)_x \text{ contains an open interval}\}$.
We have $\pi_1(\mathcal I(X)) \subset T$.
In fact, take $x \in \mathcal I(X)$ and open interval $J \subset X_x$. 
The set $\pi_1(x) \times J$ is contained in $\Pi(X)_x$.
It means that $\pi_1(x) \in T$.
We get $\myint(T) \not=\emptyset$ because $\pi_1(\mathcal I(X))$ has a nonempty interior.
Set $T_{r,s}=\{x \in R^d\;|\;\exists t \in R,\ [t-s,t+s] \subset (\Pi(X_{r,s}))_x\}$.
The set $T$ is $\myds$ and $T=\bigcup_{r,s}T_{r,s}$ as demonstrated previously.
We have $\myint(T_{r,s}) \not= \emptyset$ for some $r>0$ and $s>0$ by Lemma \ref{lem:bb1}.
We get $\myint(\Pi(X_{r,s})) \not=\emptyset$ by Lemma \ref{lem:interior0} and we obtain $\myint(\Pi(X)) \not=\emptyset$.
It is a contradiction to the assumption that $\mypd X = d$.
\end{proof}

\begin{lemma}\label{lem:good_point_pre}
Let $\mathcal R=(R,<,\ldots)$ be an expansion of a densely linearly ordered abelian group  with $\mathcal R \models \mydc,  \myluft$.
Let $X$ be a $\myds$-subset of $R^n$ of $\mypd(X)=d$.
Take a coordinate projection $\pi:X \rightarrow R^d$ such that $\pi(X)$ has a nonempty interior.
Then, there exists a $\myds$-subset $Z$ of $R^d$ such that $Z$ has an empty interior and the fiber $X \cap \pi^{-1}(x)$ is locally finite for any $x \in R^d \setminus Z$.
\end{lemma}
\begin{proof}
For all $1 \leq i \leq n-d$, we can take coordinate projections $\pi_i:R^{n-i+1} \rightarrow R^{n-i}$ with $\pi = \pi_{n-d} \circ \cdots \circ \pi_1$.
We may assume that $\pi_i$ are the coordinate projections forgetting the last coordinate without loss of generality.
Set $\Pi_i=\pi_i \circ \cdots \circ \pi_1$ and $\Phi_i=\pi_{n-d} \circ \cdots \circ \pi_{i+1}$.
Consider the sets $T_i = \{x \in R^{n-i}\;|\; \pi_i^{-1}(x) \cap \Pi_{i-1}(X) \text{ contains an open interval}\}$.
The sets $T_i$ are $\myds$ and we have $\mypd(T_i) < \mypd \Pi_{i-1}(X) = \mypd X = d$ by Lemma \ref{lem:interior3}.
Set $U_i=\Phi_i(T_i) \subset R^d$ for all $1 \leq i \leq n-d$.
The projection images $U_i$ are $\myds$-sets by Lemma \ref{lem:quo}(1).
We get $\myint(U_i) = \emptyset$ because $\mypd(T_i) < d$.
Set $Z=\bigcup_{i=1}^{n-d} U_i$.
It also has an empty interior by Lemma \ref{lem:bb2}.

The fiber $X \cap \pi^{-1}(x)$ is locally finite for any $x \in R^d \setminus Z$.
In fact, let $y \in R^n$ be an arbitrary point with $x=\pi(y)$.
Set $y_0=y$ and $y_i=\Pi_i(y)$ for $1 \leq i \leq n-d$.
We have $y_{n-d}=x$ by the definition.
We construct an open box $B_i$ in $R^{n-d-i}$ for $0 \leq i \leq n-d$ such that $y_i \in B_i$ and $(\{x\} \times B_i) \cap \Pi_i(X)$ consists of at most one point in decreasing order.
When $i=n-d$, the open box $B_{n-d}=R^0$.
When $(\{x\} \times B_i) \cap \Pi_i(X)=\emptyset$, set $B_{i-1}=B_i \times R$.
We have $(\{x\} \times B_{i-1}) \cap \Pi_{i-1}(X)=\emptyset$.
When $(\{x\} \times B_i) \cap \Pi_i(X)\not=\emptyset$, the fiber $\Pi_{i-1}(X) \cap \pi_i^{-1}(y_i)$ is locally finite by Lemma \ref{lem:dim1}.
Therefore, there exists an open box $B_{i-1}$ in $R^{n-d+1-i}$ such that $\pi_i(B_{i-1})=B_i$, $y_{i-1} \in B_{i-1}$ and $(\{x\} \times B_{i-1}) \cap \Pi_{i-1}(X)$ consists of at most one point.
We have constructed the open boxes $B_i$ in $R^{n-d-i}$ for all $0 \leq i \leq n-d$.
The existence of $B_0$ implies that $X \cap \pi^{-1}(x)$ is locally finite.
\end{proof}

\begin{lemma}\label{lem:good_point}
Let $\mathcal R=(R,<,\ldots)$ be a densely linearly ordered structure.
Let $X$ be a definable subset of $R^n$ of dimension $d<n$.
Consider the set
\begin{align*}
\mathcal G(X)=\{x \in X\;|\; &\text{there exist a coordinate projection }\pi:R^n \rightarrow R^d \\
&\text{ and an open box } B \text{ with } x \in B \text{ such that } X \cap B \text{ is the graph}\\
&\text{ of a continuous map defined on } \pi(B)\}\text{.}
\end{align*}
It is definable and constructible.
Furthermore, we have $\dim(X \setminus \mathcal G(X)) < d$ if the following conditions are all satisfied:
\begin{itemize}
\item $\mathcal R=(R,<,+,0,\ldots)$ is an expansion of a densely linearly ordered abelian group  with $\mathcal R \models \mydc,  \myluft$.;
\item $X$ is a $\myds$ set;
\item Any $\myds$ set of dimension smaller than $d$ is constructible.
\end{itemize}
\end{lemma}
\begin{proof}
It is obvious that the set $\mathcal G(X)$ is definable.
We show that $\mathcal G(X)$ is constructible.
Fix an arbitrary coordinate projection $\pi:R^n \rightarrow R^d$.
Consider the set $\mathcal G(X)_{\pi}$ of points $x \in R^n$ such that $X \cap B$ is the graph of a continuous map defined on $\pi(B)$ for some open box $B$ containing the point $x$.
The set $\mathcal G(X)_{\pi}$ is locally closed because $\mathcal G(X)_{\pi}$ is locally the graph of a continuous map.
Therefore, it is constructible.
The set $\mathcal G(X)$ is also constructible because we have $\mathcal G(X)=\bigcup_{\pi}\mathcal G(X)_{\pi}$.

The difference $X \setminus \mathcal G(X)$ is $\myds$ by Lemma \ref{lem:quo}.
We next demonstrate that $\dim(X  \setminus \mathcal G(X)) < d$ under the given conditions.
When $d=0$, $X = \mathcal G(X)$ by Lemma \ref{lem:dim0}.
The lemma is obvious.

Consider the case in which $d>0$.
Note that we always have $\mathcal G(X) \cap U = \mathcal G(X \cap U)$ for any definable open subset $U$ of $R^n$ by the definition of $\mathcal G(X)$.
We also have $U \cap (X \setminus \mathcal G(X)) = (X \cap U) \setminus \mathcal G(X \cap U)$.
We use this fact without mentioning.
Set $Y=X  \setminus \mathcal G(X)$.
We lead to a contradiction assuming that $\dim(Y) = d$.
Take a point $y \in R^n$ with $\dim_yY=d$. 
We can take a coordinate projection $\pi:R^n \rightarrow R^d$ such that we have $\pi(Y \cap U)$ has a nonempty interior for any open box $U$ containing $y$.
We fix a sufficiently small open box $B$ in $R^n$ with $y \in B$.
Note that $\dim_yX=d$ because $d=\dim_yY \leq \dim_yX \leq \dim X=d$.
We can prove by an easy induction that there exists a positive integer $N_1$ with $|X \cap B \cap \pi^{-1}(x)| \leq N_1$ or $|X \cap B \cap \pi^{-1}(x)| = \infty$ for all $x \in \pi(X \cap B)$ using Lemma \ref{lem:uniform2_2}.
We omit the proof.
We also have $\dim(X \cap B)=\dim(Y \cap B)=\mypd(X \cap B)=\mypd(Y \cap B)=d$ by Lemma \ref{lem:dim_lem1}.
We may assume that 
\begin{itemize}
\item $|X \cap \pi^{-1}(x)| \leq N_1$ or $|X  \cap \pi^{-1}(x)| = \infty$ for all $x \in \pi(X)$ and
\item $d=\dim(X)=\dim(Y)=\mypd(X)=\mypd(Y)$
\end{itemize}
considering $X \cap  B$ instead of $X$.
We can take a $\myds$-subset $Z$ of $R^d$ such that $Z$ has an empty interior and $\pi^{-1}(x) \cap X$ is locally finite for any $x \in R^d \setminus Z$ by Lemma \ref{lem:good_point_pre}.
Since $\dim(Z)<d$, $Z$ is a constructible set by the assumption.
Therefore, the sets $\pi(Y) \setminus Z$ is a $\myds$-set by Lemma \ref{lem:quo}(3).
The set $\pi(Y) \setminus Z$ has a nonempty interior by Lemma \ref{lem:bb2}.
We may further assume that 
\begin{itemize}
\item $\pi(Y)$ has a nonempty interior and 
\item $|X \cap \pi^{-1}(x)| \leq N_1$ for all $x \in \pi(X)$
\end{itemize}
considering $X \cap \pi^{-1}(B')$ instead of $X$, where $B'$ is an open box contained in $\pi(Y) \setminus Z$.

We can reduce to the case in which there exists a positive integer $N$ with $|Y \cap \pi^{-1}(x)|=N$ for all $x \in \pi(Y)$.
We need the following claim:
\medskip

\textbf{Claim.} There exists a positive integer $N$ with $N \leq N_1$ such that the set $E=\{x \in R^d\;|\; |\pi^{-1}(x) \cap Y|=N\}$ is $\myds$ and $\myint(E) \not= \emptyset$.
\medskip

We begin to prove the claim.
For all $1 \leq i \leq N_1$, consider the sets $C_i = \{x \in R^d\;|\; |\pi^{-1}(x) \cap Y|=i\}$ and $D_i = \{x \in R^d\;|\; |\pi^{-1}(x) \cap Y| \geq i\}$.
The set $D_i$ is $\myds$.
In fact, $D_i$ is the projection image of the $\myds$-set $\{(x_1,\ldots, x_i) \in (R^n)^i\;|\; \pi(x_1)=\cdots=\pi(x_i), \ x_j \not=x_k \text{ for all }j \not=k,\  x_j \in Y \text{ for all }j\}$.
We demonstrate the claim by the induction on $N_1$.
When $N_1=1$, we have nothing to prove. 
We have $N=1$ and $E=\pi(Y)$.
When $N_1>1$, the set $C_{N_1}=D_{N_1}$ is $\myds$.
If $\myint(C_{N_1}) \not=\emptyset$, set $N=N_1$ and $E=C_{N_1}$.
Otherwise, we have $\dim C_{N_1} <d$.
The definable set $C_{N_1}$ is constructible by the assumption.
 By the induction hypothesis, we can get $N < N_1$ such that $E=\{x \in R^d\;|\;  |\pi^{-1}(x) \cap (Y \setminus (C_{N_1} \times R^{n-d}))|=N\}$ is $\myds$ and $\myint(E) \not= \emptyset$.
 It is obvious that $E=\{ x \in R^d\;|\;  |\pi^{-1}(x) \cap Y |=N\}$.
 We have proven the claim.
 \medskip

Let $E$ be the $\myds$-set in the claim and take an open box $B''$ contained in $E$.
Set $X'=X \cap \pi^{-1}(B'')$.
We may assume that $|Y \cap \pi^{-1}(x)|=N$ for all $x \in \pi(X)$ and $\pi(X)$ is an open box by considering $X'$ in place of $X$.
Applying the same argument to the new $X$, we can reduce to the following case:
\begin{itemize}
\item We have $\pi(Y)=\pi(X)=V$ for some open box $V$ in $R^d$;
\item There exist positive integers $N$ and $N'$ such that $|Y \cap \pi^{-1}(x)| =N$ and $|X \cap \pi^{-1}(x)| =N'$ for any $x \in V$.
\end{itemize}
We demonstrate that the closure of $\mathcal G(X)$ has an empty intersection with $Y$.
Let $y$ be a point in the intersection.
Let $\{y_1, \ldots, y_M\}$ be the fiber $\mathcal G(X) \cap \pi^{-1}(\pi(y))$, where $M=N'-N$.
We have $y \not=y_i$ for all $1 \leq i \leq M$ and $y_i \not=y_j$ for $i \not=j$.
Since $\mathcal G(X)$ is locally the graph of a continuous function, there are $y'_1, \ldots, y'_M \in \mathcal G(X)$ such that $\pi(y'_1)=\cdots=\pi(y'_M)$ and $y'_i$ are sufficiently closed to $y_i$ for $1 \leq i \leq M$.
Since $y$ is a point of the closure, there exists $y' \in \mathcal G(X)$ sufficiently close to $y$ with $\pi(y')=\pi(y'_1)$.
The fiber $\mathcal G(X) \cap \pi^{-1}(\pi(y'))$ contains the $(M+1)$ points $y'_1, \ldots, y_M$ and $y'$.
Contradiction to the fact that $|\mathcal G(X) \cap \pi^{-1}(\pi(y'))|=M$.

 Set $$Y_i=\{x \in R^n\;|\; x \text{ is the } i \text{-th minimum in }Y \cap \pi^{-1}(\pi(x))\text{ in the lexicographic order}\}$$ for all $1 \leq i \leq N$.
Consider the $\myds$-set $Z=\{(x_1, \ldots, x_N) \in (R^n)^N\;|\; \pi(x_1) = \cdots = \pi(x_N),\  x_1 < x_2 < \cdots <x_N \text{ in the lexicographic order}, \ x_i \in X \text{ for all } 1 \leq i \leq N\}$.
The definable set $Y_i$ is the projection image of the $\myds$-set $Z$, and it is also $\myds$ by Lemma \ref{lem:quo}(1).
The projection image $\pi(Y_i)$ is an open box because $\pi(Y_i)=\pi(Y)=V$.
Consequently, $Y_i$ is simultaneously a $\myds$-set and the graph of a definable map defined on an open box for any $1 \leq i \leq N$.
Applying Lemma \ref{lem:cont_function} iteratively to $Y_i$, we can find a nonempty open box $W$ such that $Y_i \cap \pi^{-1}(W)$ are the graphs of definable continuous maps defined on $W$ for all $1 \leq i \leq N$.
Since the closure of $\mathcal G(X)$ has an empty intersection with $Y$, we have $\widetilde{B} \cap Y=\widetilde{B} \cap X$ for any point $y \in Y$ and any sufficiently small open box $\widetilde{B}$ containing the point $y$.
 We have $Y \cap \pi^{-1}(W) \subset \mathcal G(X)$.
 Contradiction to the definition of $Y$.
\end{proof}

We finally get the following theorem.
\begin{theorem}\label{thm:main2}
Let $\mathcal R=(R,<,+,0,\ldots)$ be an expansion of a densely linearly ordered abelian group  with $\mathcal R \models \mydc,  \myluft$.
Any $\myds$-set is constructible.
In particular, any set definable in the open core of $\mathcal R$ is constructible.
\end{theorem}
\begin{proof}
Let $X$ be a $\myds$-subset of $R^n$ of dimension $d$.
We show that $X$ is constructible by the induction on $d$.
When $d=0$, it is clear from Lemma \ref{lem:dim0}.
When $d>0$, consider the constructible set $\mathcal G(X)$ defined in Lemma \ref{lem:good_point}.
The difference $X \setminus \mathcal G(X)$ is a $\myds$-set of dimension smaller than $d$ by Lemma \ref{lem:good_point}.
It is constructible by the induction hypothesis.
Consequently, $X$ is also constructible.

It is obvious that any set definable in the open core of $\mathcal R$ is constructible because the projection image of a constructible set is again constructible by the assertion we have just proven and Lemma \ref{lem:quo}(1).
\end{proof}

\begin{proof}[Proof of Theorem \ref{thm:main}]
Theorem \ref{thm:main} is now obvious by Lemma \ref{lem:quo}(3), Lemma \ref{lem:uniform2_2} and Theorem \ref{thm:main2}.
\end{proof}

\end{document}